\documentclass[reqno,11pt]{amsart} 

\usepackage{amsmath,amsfonts,latexsym,amscd,amssymb,bbm}
\usepackage{mathrsfs}       
\usepackage{hyperref}
\usepackage[nameinlink]{cleveref}
\usepackage{xcolor}         
\usepackage{tikz}           
\usepackage{verbatim}       
\usepackage{setspace}       
\usepackage[normalem]{ulem}

\definecolor{purple}{rgb}{0.65, 0, 1}
\definecolor{orange}{rgb}{1,.5,0}
\definecolor{brown}{rgb}{.9,.73,.26}

\newtheorem{theorem}{Theorem}[section]
\newtheorem{remark}{Remark}[section]
\newtheorem{proposition}[theorem]{Proposition}
\newtheorem{definition}{Definition}[section]
\newtheorem{lemma}[theorem]{Lemma}

\numberwithin{equation}{section}

\def\R{{\mathbb {R}}}

\newcommand{\abs}[1]{\lvert#1\rvert}
\usepackage{color}

\def\R{\mathbb{R}}

\title[non-existence  to the Grushin heat equation]{Nonexistence of global solutions to the Grushin heat equation with nonlocal and local nonlinearities}

\author[A. Fino]{Ahmad Z. Fino$^{*,\lowercase{a}}$}
\address{$^{*,a}$College of Engineering and Technology, American University of the Middle East, Kuwait}
\email{ahmad.fino@aum.edu.kw}

\author[A. Viana]{Arlúcio Viana$^{\lowercase{b,c}}$}
\address{$^{b}$Department of Mathematics, Universidade Federal de Sergipe, Brazil}
\address{$^{c}$Institut f\"ur Angewandte Analysis, Universit\"at Ulm, Germany.}
\email{arlucioviana@academico.ufs.br}

\begin{document}
	





\begin{abstract}
In this paper, we investigate the nonexistence of global solutions to the Grushin-type heat equation with nonlinear reaction terms, including cases involving memory effects:
	\begin{equation*}
		\left\{\begin{array}{ll}
			\displaystyle {u_{t}-\Delta_{\mathcal{G}}
				u = k_1 \int_0^t(t-s)^{-\gamma}\abs{u}^{p_1-1}u(s)\,\mathrm{d}s} + k_2|u|^{p_2-1}u, &  (z,t)\in {\mathbb{R}}^{N+k}\times (0,\infty),\\
			\displaystyle{u(z,0)=  u_0(z),\qquad\qquad}&\displaystyle{z\in {\mathbb{R}}^{N+k},}
		\end{array}
		\right.
	\end{equation*}
	where $\Delta_{\mathcal{G}}$ denotes the Grushin operator, $u_0 \in L^1_{\mathrm{loc}}(\mathbb{R}^{N+k})$, $\gamma\in[0,1)$, $k_1,k_2 \geq 0$, and $p_1,p_2>1$.
	We establish sharp nonexistence results for global-in-time positive solutions, thereby completing the picture of global existence versus blow-up and allow us to identify the corresponding Fujita-type critical exponents in certain parameter regimes.
The analysis relies on the test function method, adapted to handle both the degeneracy of the Grushin operator and the influence of the memory term.

\end{abstract}

\date{\today}

\maketitle

\section{Introduction}

In this paper, we address the non-existence of global solutions for the Grushin heat equation with mixed nonlinear memory and reaction terms:
\begin{equation}\label{*}
	\left\{\begin{array}{ll}
		\displaystyle {u_{t}-\Delta_{\mathcal{G}}
			u = k_1 \int_0^t(t-s)^{-\gamma}\abs{u}^{p_1-1}u(s)\,\mathrm{d}s} + k_2|u|^{p_2-1}u, &  (z,t)\in {\mathbb{R}}^{N+k}\times (0,\infty),\\
		\displaystyle{u(z,0)=  u_0(z),\qquad\qquad}&\displaystyle{z\in {\mathbb{R}}^{N+k},}
	\end{array}
	\right.
\end{equation}
where $u_0 \in L^1_{\mathrm{loc}}(\mathbb{R}^{N+k})$, $\gamma\in[0,1)$, $k_1,k_2\geq 0$, and $p_1,p_2>1$. The  Grushin operator is defined by
\begin{equation*}\label{Grushiop}
	\Delta_{\mathcal{G}}= \Delta_{x}+|x|^2\Delta_{y}, 
\end{equation*}
where $\Delta_x, \Delta_y$ denote the classical Laplacian in the variables $x\in\R^N$ and $y\in\R^k$, respectively.

Diffusion equations with nonlinear memory reaction, as well as the combination of with and without memory reactions, are well-established subjects in the literature, see e.g. \cite{CDW,Be-87, deAn-Vi-17, Webb-78, Yam-88, Yam-84} and references therein. Existence of solutions and other qualitative properties for the heat equation mixed pure and memory nonlinearities are also discussed in several frameworks, see \cite{Ca-Vi-15, deAn-Vi-17, Gladkov-24, Soup-96, Zh-23}. Furthermore, mixed power-type nonlinearities have been largely studied (see e.g. \cite{Br-Ni-83}, references therein and thereof). For evolution problems, we refer to \cite{Alves-Bou-24, Fino-C-20, FVR-11, Loayza-06, Sli-17}. Those works are dedicated to well-posedness, global existence of positive solutions, and finite-time blow-up of the sign-changing. In fact,  nonlocal nonlinearities play an important role in the analysis of the evolution equations, for example, dissociating the critical Fujita exponent from the one given by scaling computations, which was for the first time observed by Cazenave, Dickstein and Weissler \cite{CDW}. They proved that the critical exponent for 
$$ u_t-\Delta u=\int_0^t(t-s)^{-\gamma}|u|^{p-1} u(s) d s \quad x \in \mathbb{R}^N, t>0 ,$$
where $0 \leq \gamma<1$, $p>1$, is $p_*=\max \left\{\frac{1}{\gamma}, p_\gamma\right\}$, where
\begin{equation}\label{FC-CDW}
	p_\gamma=1+\frac{2(2-\gamma)}{(N-2+2 \gamma)_{+}} ,	
\end{equation}
while the exponent predicted by the scaling is $p_{\mathrm{sc}}=1+\frac{4-2 \gamma}{N}$. Subsequently, Fino and Kirane generalized that result  for systems in \cite{Fino-Ki-11}, and by replacing the Laplacian with the fractional Laplacian in \cite{Fino-Ki-12}. Then this type of investigation attracted more researchers. See e.g. \cite{Cas-Loy-19,D-Ki-Fino-18,Ki-La-T-05,Loy-P-14}. 

On the other hand, the hypoelliptic operator introduced by Baouendi and Grushin \eqref{Grushiop} \cite{Baouendi-67, Grushin} has attracted the interest of many researchers in evolution equations in recent years, cf. \cite{Alves-G-Lo-24, A-14, Calin-11, Chang-Li-12,Gaveau-77,Goldstein-Kogoj, Kogoj-Lan-12,Wu-15}. Recently, there has been an increasing interest in nonlinear parabolic equations with sub-elliptic operators, see e.g. the recent papers \cite{A-14, Bu-Lat-25, Fino-Ruz-Tor-24, Hi-Oka-24, J-T-14, J-Ki-S-16,RuzY-22}. At the core of some analysis of heat equations with the Grushin operator is the study of the heat kernel associated with the Grushin operator, see e.g. \cite{Garofalo-Trallli-22, Oli-Vi-23, Stempak-25}.

When $k_1=0$ in \eqref{*}, Oliveira and Viana \cite{Oli-Vi-23} proved the existence, uniqueness, continuous dependence, and blowup alternative of local mild solutions,
with initial conditions taken in Lebesgue spaces. Through their analysis, they also obtained the existence of global solutions in the special case of $u_0\in L^{\frac{ N + 2k}{2}(\rho-1)}(\R^{N+k})$ with a sufficiently small norm. Later on, Kogoj, Lima and Viana \cite{Kogoj-Lima-Viana-24} by working in Marcinkiewicz  spaces, they allowed larger (in Lebesgue norm) initial data to be taken into account in order to obtain global solutions. More precisely, they gave sufficient conditions for the existence and uniqueness of mild solutions for \eqref{*} ($k_1=0)$, with initial conditions in the critical Marcinkiewicz  space $L^{\frac{N+2k}{2}(\rho-1) ,\infty}(\R^{N+k})$. Moreover, they obtained the existence of positive, symmetric, and self-similar solutions. The number $N+2k$ is the so-called homogeneous dimension attached to Grushin's operator (see \cite{Kogoj-Lan-18}). More recently, the authors proved the existence of global nonnegative solutions, cf. \cite{Fino-Viana-25.1}. Their approach required well-posedness and comparison results, which were also developed in the same paper.

Non-existence results for heat equations with combined pure and memory reaction have been studied by Souplet \cite{Soup-96}, in the case of the Laplace operator, where the author considered the equation
\begin{equation}\label{Souplet}
	u_t-\Delta u=\mu(x) \int_0^t u^p(s, x) d s-a u^q(t, x) \quad t>0, \quad x \in \Omega,
\end{equation}
where $p, q \geq 1, a>0, \mu$ is Hölder continuous in $\bar{\Omega}, \mu \geq 0, \mu \not \equiv 0$. Under suitable conditions, he proved that blow up occurs whenever $p>q$ and for all nonnegative nontrivial $u_0$., and $p=q$ is proved to be the critical blow up exponent. Recently, Zhang \cite{Zh-23} generalized the results by replacing in \eqref{Souplet} the usual time-derivative with the Caputo fractional one and considering a singular kernel in the memory reaction term. However, when addressing heat equations involving non-classical operators as the Grushin operator, such results appear to be lacking in the literature.

We analyze the non-existence of nonnegative solutions for \eqref{*} in different cases. Combined with the results in \cite{Fino-Viana-25.1}, our results provide the critical Fujita exponent \cite{Fujita} in some specific cases. More precisely,
\begin{enumerate}
	\item For $k_1=0$, the critical Fujita exponent for \eqref{*} is given by
	$$
	p_{c_1} := 1+\frac{2}{N+2k} .
	$$
	See Theorem \ref{T1} below. We observe that this is compatible with the classical \textit{Euclidean case}, since $N+2k$ is the homogeneous dimension in the space induced by the Grushin geometry.

	\item For $k_2=0$, the same phenomenon occurs: the critical Fujita exponent for \eqref{*} is given by
	$$p_{c_2} := \max\left\{ 1+\frac{2(2-\gamma)}{N+2k-2+2\gamma} , \frac{1}{\gamma} \right\},$$
	which is compatible with \eqref{FC-CDW}. See Theorems \ref{T2} and \ref{T3} below.

	\item For $k_1,k_2>0$, the condition for non-existence of nonnegative solutions (see Theorems \ref{T4} and \ref{T5}) lies on
	$$p_1 \leq p_{c_1} \ \ \mbox{or} \ \ p_2 \leq p_{c_2} . $$
	
\end{enumerate} 

The paper is organized as follows. Section \ref{pre} introduces key definitions, terminology, and preliminary results. In Section \ref{nonex}, we state and prove the precise non-existence results for \eqref{*}.

\section{Preliminaries}\label{pre}
In this section, we gather the definitions and results used in our proofs. They are related to tools from the Fractional Calculus and integrodifferential inequalities.

\begin{definition}
	A function $\mathcal{A} :[a,b]\rightarrow\mathbb{R}$, $-\infty<a<b<\infty$, is said to be absolutely continuous if and only if there exists $\psi\in L^1(a,b)$ such that
	\[
	\mathcal{A} (t)=\mathcal{A} (a)+\int_{a}^t \psi(s)\,ds.
	\]
	$AC[a,b]$ denotes the space of these functions.  Moreover,
\[
	AC^2[a,b]:=\left\{\varphi:[a,b]\rightarrow\mathbb{R};\,\varphi'\in AC[a,b]\right\}.
	\]
\end{definition}

The following definitions and properties of fractional operators can be found in \cite[Chapter~1]{Samko}.
\begin{definition}[Riemann-Liouville fractional integral]
	Let $f\in L^1(c,d)$, $-\infty<c<d<\infty$. The Riemann-Liouville left-and right-sided fractional integrals are defined by
	\begin{equation}\label{Int1}
		I^\alpha_{c|t}f(t):=\frac{1}{\Gamma(\alpha)}\int_{c}^t(t-s)^{-(1-\alpha)}f(s)\,ds, \quad t>c, \; \alpha\in(0,1),
	\end{equation}
	and
	\begin{equation}\label{Int2}
		I^\alpha_{t|d}f(t):=\frac{1}{\Gamma(\alpha)}\int_t^{d}(s-t)^{-(1-\alpha)}f(s)\,ds, \quad t<d, \; \alpha\in(0,1)
	\end{equation}
	where $\Gamma$ is the Euler gamma function.
\end{definition}

\begin{definition}[Riemann-Liouville fractional derivative]
	Let $f\in AC[c,d]$, $-\infty<c<d<\infty$. The Riemann-Liouville  left-and right-sided fractional derivatives  are defined by
	\begin{equation}\label{}
		D^\alpha_{c|t}f(t):=\frac{d}{\mathrm{dt}}I^{1-\alpha}_{c|t}f(t)=\frac{1}{\Gamma(1-\alpha)}\frac{d}{\mathrm{dt}}\int_{c}^t(t-s)^{-\alpha}f(s)\,ds, \quad t>c,\quad \alpha\in(0,1),
	\end{equation}
	$$
	D^\alpha_{c|t}f(t):=\frac{1}{\Gamma(2-\alpha)}\frac{d^2}{\mathrm{dt}^2}\int_{c}^t(t-s)^{-(\alpha-1)}f(s)\,ds, \quad t>c,\quad\alpha\in(1,2),
	$$
	\begin{equation}\label{}
		D^\alpha_{t|d}f(t):=-\frac{d}{\mathrm{dt}}I^{1-\alpha}_{t|d}f(t)=-\frac{1}{\Gamma(1-\alpha)}\frac{d}{\mathrm{dt}}\int_t^{d}(s-t)^{-\alpha}f(s)\,ds, \quad t<d,\quad \alpha\in(0,1),
	\end{equation}
	and
	$$
	D^\alpha_{t|d}f(t):=\frac{1}{\Gamma(2-\alpha)}\frac{d^2}{\mathrm{dt}^2}\int_t^{d}(s-t)^{-(\alpha-1)}f(s)\,ds, \quad t<d,\quad\alpha\in(1,2).
	$$
\end{definition}

For the next proposition, see \cite[(2.64), p. 46]{Samko}
\begin{proposition}
	Let $\alpha\in(0,1)$ and $-\infty<c<d<\infty$. The fractional integration by parts formula
	\begin{equation}\label{IP}
		\int_{c}^{d}f(t)D^\alpha_{c|t}g(t)\,\mathrm{dt} \;=\; \int_{c}^{d}
		g(t)D^\alpha_{t|d}f(t)\,\mathrm{dt},
	\end{equation}
	is satisfied for every $f\in I^\alpha_{t|d}(L^p(c,d))$, $g\in I^\alpha_{c|t}(L^q(c,d))$ such that $\frac{1}{p}+\frac{1}{\R^{N+k}}\leq 1+\alpha$, $p,q>1$, where
	\[
	I^\alpha_{c|t}(L^q(0,T)):=\left\{f= I^\alpha_{c|t}h,\,\, h\in L^q(c,d)\right\},
	\]
	and
	\[
	I^\alpha_{t|d}(L^p(c,d)):=\left\{f= I^\alpha_{t|d}h,\,\, h\in L^p(c,d)\right\}.
	\]
\end{proposition}
\begin{remark}
	A simple sufficient condition for functions $f$ and $g$ to satisfy (\ref{IP}) is that $f,g\in C[c,d],$ such that
	$D^\alpha_{t|d}f(t),D^\alpha_{c|t}g(t)$ exist at every point $t\in[c,d]$ and are continuous.
\end{remark}

\begin{proposition}
	For $0<\alpha<1$, $-\infty<c<d<\infty$, we have the following identities
	\begin{equation}\label{Int3}
		D^\alpha_{c|t}I^\alpha_{c|t}\varphi(t)=\varphi(t),\,\hbox{a.e. $t\in(c,d)$}, \quad\hbox{for all}\,\,\varphi\in L^r(c,d),\,\, 1\leq r\leq\infty,
	\end{equation}
	and
	\begin{equation}\label{Int4}
		-\frac{d}{\mathrm{dt}} \, D^\alpha_{t|d}\varphi=D^{1+\alpha}_{t|d}\varphi,\quad \varphi\in AC^2[c,d].
	\end{equation}
\end{proposition}

Given $T>0$, let us define the function $w_1$ by
\begin{equation}\label{w1}
	\displaystyle w_1(t)=\left(1-t/T\right)_+^\sigma,\quad \sigma\gg1
\end{equation}
%
%
%
%
%

\begin{lemma}\cite[(2.45), p. 40]{Samko}\label{lemma1} \\
	Let $T>0$, $0<\alpha<1$ and $m\geq0$. For all $t\in[0,T]$, we have
	\begin{equation}\label{Int6}
		D_{t|T}^{m+\alpha}w_1(t)=\frac{\Gamma(\sigma+1)}{\Gamma(\sigma+1-m-\alpha)}T^{-(m+\alpha)}(1-t/T)^{\sigma-\alpha-m}.
	\end{equation}
\end{lemma}

\begin{lemma}\label{lemma2}
	Let $T>0$, $0<\alpha<1$, $m\geq0$, $p>1$, and
	$$\sigma > \frac{p(\alpha+m-1) + 1}{p-1} .$$ 
	Then, we have
	\begin{equation}\label{Int9}
		\int_0^T \big(w_1(t)\big)^{-\frac{1}{p-1}}|D_{t|T}^{m+\alpha}w_1(t)|^{\frac{p}{p-1}}\,\mathrm{dt}=C\,T^{1-(m+\alpha)\frac{p}{p-1}},
	\end{equation}
	and
	\begin{equation}\label{Int8}
		\int_0^TD_{t|T}^{m+\alpha}w_1(t)\,\mathrm{dt}= C\,T^{-(m+\alpha)}.
	\end{equation}
\end{lemma}

\begin{proof}
	Using Lemma \ref{lemma1} we have
	\begin{align*}
		\int_{0}^T \big(w_1(t)\big)^{-\frac{1}{p-1}}|D_{t|T}^{m+\alpha}w_1(t)|^{\frac{p}{p-1}}\,\mathrm{dt} &= C\,T^{-(m+\alpha)\frac{p}{p-1}}\int_{0}^T \big(w_1(t)\big)^{-\frac{1}{p-1}}\big(w_1(t)\big)^{\frac{p({\sigma}-\alpha-m)}{(p-1)\sigma}}\,\mathrm{dt}\\
		&= C\,T^{-(m+\alpha)\frac{p}{p-1}}\int_{0}^T(1-t/T)^{\sigma-(m+\alpha)\frac{ p}{p-1}}\,\mathrm{dt}\\
		&= C\,T^{1-(m+\alpha)\frac{p}{p-1}}\int_{0}^1(1-s)^{\sigma-(m+\alpha)\frac{ p}{p-1}}\,ds\\
		&= C\,T^{1-(m+\alpha)\frac{p}{p-1}},
	\end{align*}
	where we have used $\sigma \gg1$ to guarantee the integrability of the last integral. We obtain \eqref{Int8} in the same way.
\end{proof}

\begin{lemma}\cite[Proposition~2.2]{CDW}\label{lemma3} \\
	Let $T>0$, $0\leq \gamma<1$, $p>1$, $a>0$, and $b>0$, there exists a constant $K=K(a,b,p)$ such that, if $w\in C^1([0,T])$, $w>0$, and satisfies 
	$$w'(t)+a\,w(t)\geq b\int_0^t(t-s)^{-\gamma}w^p(s)\,ds,\quad\hbox{for all}\,\,T>0,$$
	then the following properties hold:
	\begin{itemize}
		\item[$\mathrm{(i)}$] If $T\geq 1$, then $w(0)<K$.
		\item[$\mathrm{(ii)}$] If $p\gamma\leq 1$, then $T<\infty$.
	\end{itemize}
\end{lemma}

\begin{lemma}\label{lemma4}
	For $\varepsilon>0$, $A>0$, $c>0$, let
	
	$$\Theta(z)=c\,e^{-\varepsilon\sqrt{A+|x|^4+|y|^2}},\qquad z=(x,y)\in \mathbb{R}^{N+k}.$$
	Then
	$$
	\Delta_{\mathcal{G}}\Theta(z)\geq -\varepsilon[2(N+2)+k]\Theta(z),\qquad\hbox{for all}\,\,z\in\mathbb{R}^{N+k}.
	$$
	${}$
\end{lemma}
\begin{proof} Let $z=(x,y)\in \mathbb{R}^{N+k}$, and $\rho(z):=A+|x|^4+|y|^2$. We have
	$$\partial_{z_i} \Theta(z)=-\frac{\varepsilon}{2}\rho^{-\frac{1}{2}}\partial_{z_i}(\rho)\Theta(z)
	$$
	and then
	$$
	\partial_{z_i}^2 \Theta(z)=\frac{\varepsilon}{4}\rho^{-\frac{3}{2}}|\partial_{z_i}(\rho)|^2\Theta(z)-\frac{\varepsilon}{2}\rho^{-\frac{1}{2}}\partial_{z_i}^2(\rho)\Theta(z)+\frac{\varepsilon^2}{4}\rho^{-1}|\partial_{z_i}(\rho)|^2\Theta(z) . 
	$$
	Hence
	$$
	\Delta_x \Theta(z)=\frac{1}{4}\left(\varepsilon\rho^{-\frac{3}{2}}+\varepsilon^2\rho^{-1}\right)|\nabla_x(\rho)|^2\Theta(z)-2\varepsilon (N+2)|x|^2\rho^{-\frac{1}{2}}\Theta(z),
	$$
	and
	$$
	\Delta_y \Theta(z)=\frac{1}{4}\left(\varepsilon\rho^{-\frac{3}{2}}+\varepsilon^2\rho^{-1}\right)|\nabla_y(\rho)|^2\Theta(z)-\varepsilon k\,\rho^{-\frac{1}{2}}\Theta(z).
	$$
	As $\nabla_x(\rho)=4x|x|^2$, $\Delta_x(\rho)=4(N+2)|x|^2$, $\nabla_y(\rho)=2y$, and $\Delta_y(\rho)=2k$, we conclude that
	\begin{equation}\label{36}
		\begin{split}\Delta_{\mathcal{G}}\Theta(z)&=\Delta_x \Theta(z)+|x|^2\Delta_y \Theta(z)\\
			&=\frac{1}{4}\left(\varepsilon\rho^{-\frac{3}{2}}+\varepsilon^2\rho^{-1}\right)\left(|\nabla_x(\rho)|^2+|x|^2|\nabla_y(\rho)|^2\right)\Theta(z)-\varepsilon[2(N+2)+k]|x|^2\rho^{-\frac{1}{2}}\Theta(z)\\
			&\geq -\varepsilon[2(N+2)+k]|x|^2\rho^{-\frac{1}{2}}\Theta(z).
		\end{split}
	\end{equation}
	Finally, using the fact that $|x|^4\leq \rho$ implies $|x|^2 \rho^{-\frac{1}{2}}\leq 1$, we get
	$$\Delta_{\mathcal{G}}\Theta(z)\geq-\varepsilon[2(N+2)+k]\Theta(z).$$ 
	This completes the proof.
\end{proof}

\section{Non-existence of global solutions}\label{nonex}
In this section, we present and prove the non-existence results, which are categorized according to the cases $k_1=0,\ k_2=1$; $k_1=1,\ k_2=0$; and $k_1= k_2=1$. These results complete the characterization of the existence of global positive solutions (cf. \cite{Fino-Viana-25.1}) in each respective case. Additional non-existence results are established at the end of the section. Our approach is based on the notion of weak solutions. Throughout this section write $z=(x,y)\in\mathbb{R}^{N+k}$.

\begin{definition}[Weak solution]
	Let $u_0\in L_{\mathrm{loc}}^1(\R^{N+k})$, and $T>0$. We say that $u$ is a
	weak solution of the problem \eqref{*} if 
	$$I^\alpha_{0|t}(|u|^{p_1-1}u), |u|^{p_2-1}u \in L_{\mathrm{loc}}^1((0,T),L_{\mathrm{loc}}^1(\R^{N+k})) , $$
	and verifies the weak formulation
	\begin{eqnarray}\label{weaksolution}
		&&\int_{\R^{N+k}}u(z,\tau)\varphi(z,\tau)\,\mathrm{dz}-\int_{\R^{N+k}}u_0(z)\varphi(z,0)\,\mathrm{dz}\nonumber\\
		&&=k_1\Gamma(\alpha)\int_0^T\int_{\R^{N+k}}I^\alpha_{0|t}(|u|^{p_1-1}u)\varphi(z,t)\,\mathrm{dz}\,\mathrm{dt}+k_2\int_0^T\int_{\R^{N+k}}|u|^{p_2-1}u\varphi(z,t)\,\mathrm{dz}\,\mathrm{dt}\nonumber\\
		&&+\int_0^T\int_{\R^{N+k}}u(z,t)\Delta_{\mathcal{G}}\varphi(z,t)\,\mathrm{dz}\,\mathrm{dt}+\int_0^T\int_{\R^{N+k}}u(z,t)\varphi_t(z,t)\,\mathrm{dz}\,\mathrm{dt},
	\end{eqnarray}
	for all compactly supported $\varphi\in C^{2,1}(\R^{N+k}\times[0,T])$ such that $\varphi(T,\cdotp)=0$, where $\alpha:=1-\gamma\in(0,1)$.
\end{definition}

The next lemma relax the assumptions on the test function.

\begin{lemma}\label{lemma} 
	Let $u_0\in L^\infty(\R^{N+k})$, and $T>0$. Let $\psi_0\in C^{2,1}( \R^{N+k}\times[0,T))$ be such that
	\begin{equation}\label{49}
		\int_{ \R^{N+k}}\{|\psi_0(z,t)|+|\partial_t\psi_0(z,t)|+|\nabla_{x}\psi_0(z,t)|+|\nabla_{y}\psi_0(z,t)|+|\Delta_{\mathcal{G}}\psi_0(z,t)|\}\,dz<\infty,
	\end{equation}
	for all $t\in[0,T)$. If $ u\in  L_{\mathrm{loc}}^\infty((0,T);L^\infty(\R^{N+k}))$ is a weak solution of \eqref{*}, then 
	\begin{eqnarray*}
		&&\int_{\R^{N+k}}u(z,\tau) \psi_0(z,\tau)\,\mathrm{dz}-\int_{\R^{N+k}}u_0(z) \psi_0(z,0)\,\mathrm{dz}\nonumber\\
		&&=k_1\Gamma(\alpha)\int_0^\tau\int_{\R^{N+k}}I^\alpha_{0|t}(|u|^{p_1-1}u) \psi_0(z,t)\,\mathrm{dz}\,\mathrm{dt}+k_2\int_0^\tau\int_{\R^{N+k}}|u|^{p_2-1}u \psi_0(z,t)\,\mathrm{dz}\,\mathrm{dt}\nonumber\\
		&&+\int_0^\tau\int_{\R^{N+k}}u(z,t)\Delta_{\mathcal{G}} \psi_0(z,t)\,\mathrm{dz}\,\mathrm{dt}+\int_0^\tau\int_{\R^{N+k}}u(z,t) \partial_t\psi_0(z,t)\,\mathrm{dz}\,\mathrm{dt},
	\end{eqnarray*}
	for all $\tau\in[0,T]$.
\end{lemma}

\begin{proof}
	Let $0<T\leq\infty$. Suppose that $ u\in  L_{loc}^\infty((0,T);L^\infty(\R^{N+k}))$ is a weak solution of \eqref{*}  on $[0,T)\times \R^{N+k}$, then we have
	\begin{eqnarray*}
		&&\int_{\R^{N+k}}u(z,\tau)\varphi(z,\tau)\,\mathrm{dz}-\int_{\R^{N+k}}u_0(z)\varphi(z,0)\,\mathrm{dz}\nonumber\\
		&&=k_1\Gamma(\alpha)\int_0^\tau\int_{\R^{N+k}}I^\alpha_{0|t}(|u|^{p_1-1}u)\varphi(z,t)\,\mathrm{dz}\,\mathrm{dt}+k_2\int_0^\tau\int_{\R^{N+k}}|u|^{p_2-1}u\varphi(z,t)\,\mathrm{dz}\,\mathrm{dt}\nonumber\\
		&&+\int_0^\tau\int_{\R^{N+k}}u(z,t)\Delta_{\mathcal{G}}\varphi(z,t)\,\mathrm{dz}\,\mathrm{dt}+\int_0^\tau\int_{\R^{N+k}}u(z,t)\varphi_t(z,t)\,\mathrm{dz}\,\mathrm{dt},
	\end{eqnarray*}
	for all compactly supported $\varphi\in C^{2,1}(\R^{N+k}\times[0,T))$, and $0\leq\tau<T$. Let $\tau\in[0,T)$ be a fixed number, and let
	$$\varphi(z,t):= \varphi_R(z)\varphi_3(t)\psi_0(z,t):=\varphi_1(x)\varphi_2(y)\varphi_3(t) \psi_0(z,t),\quad t\in[0,\tau),\,\,z=(x,y)\in Q,$$
	with
	$$\varphi_1(x):=\Phi\left(\abs{x}/R\right),\quad \varphi_2(y):=\Phi\left(\abs{y}/R^3\right),\quad \varphi_3(t):=\Phi\left(t/\tau\right),$$
	where $R\gg 1$, and $\Phi\in \mathcal{C}^\infty(\mathbb{R})$ is smooth non-increasing functions 
	satisfying 
	$$\mathbbm{1}_{(-\infty,1]} \leq \Phi\leq \mathbbm{1}_{(-\infty,2]}.$$
	Then
	\begin{eqnarray*}
		&&\int_{\mathcal{B}}u(\tau,z)\varphi_R(z)\psi_0(\tau,z)\,\mathrm{dz}-\int_{\mathcal{B}}u_0(z)\varphi_R(z)\psi_0(z,0)\,\mathrm{dz}\nonumber\\
		&&=k_1\Gamma(\alpha)\int_0^\tau\int_{\mathcal{B}}I^\alpha_{0|t}(|u|^{p_1-1}u)\varphi_R(z)\psi_0(z,t)\,\mathrm{dz}\,\mathrm{dt}+k_2\int_0^\tau\int_{\mathcal{B}}|u|^{p_2-1}u\varphi_R(z)\psi_0(z,t)\,\mathrm{dz}\,\mathrm{dt}\nonumber\\
		&&+\int_0^\tau\int_{\mathcal{B}}u(z,t)\Delta_{\mathcal{G}}\left(\varphi_R(z)\psi_0(z,t)\right)\,\mathrm{dz}\,\mathrm{dt}+\int_0^\tau\int_{\mathcal{B}}u(z,t)\varphi_R(z)\partial_t\psi_0(z,t)\,\mathrm{dz}\,\mathrm{dt},
	\end{eqnarray*}
	where
	$$\mathcal{B}=\{z=(x,y)\in \R^{N+k};\,\,|x|\leq 2R,\,|y|\leq 2R^3\}.$$
	Using the identity
	$$\Delta_{\mathcal{G}}(wv)=v \Delta_{\mathcal{G}}(w)+\nabla_x(w)\cdotp\nabla_x(v)+|x|^2\nabla_y(w)\cdotp\nabla_y(v)+w\Delta_{\mathcal{G}}(v).$$
	we get
	\begin{eqnarray*}
		&&\int_{\mathcal{B}}u(\tau,z)\varphi_R(z)\psi_0(\tau,z)\,\mathrm{dz}-\int_{\mathcal{B}}u_0(z)\varphi_R(z)\psi_0(z,0)\,\mathrm{dz}\nonumber\\
		&&=k_1\Gamma(\alpha)\int_0^\tau\int_{\mathcal{B}}I^\alpha_{0|t}(|u|^{p_1-1}u)\varphi_R(z)\psi_0(z,t)\,\mathrm{dz}\,\mathrm{dt}+k_2\int_0^\tau\int_{\mathcal{B}}|u|^{p_2-1}u\varphi_R(z)\psi_0(z,t)\,\mathrm{dz}\,\mathrm{dt}\nonumber\\
		&&+\int_0^\tau\int_{\mathcal{B}}u(z,t)\psi_0(z,t)\Delta_{\mathcal{G}}\left(\varphi_R(z)\right)\,\mathrm{dz}\,\mathrm{dt} + 2\int_0^\tau\int_{\mathcal{B}}u(z,t)\nabla_x(\varphi_R(z))\cdotp\nabla_x\left(\psi_0(z,t)\right)\,\mathrm{dz}\,\mathrm{dt}\nonumber\\
		&& + 2\int_0^\tau\int_{\mathcal{B}}u(z,t)|x|^2\nabla_y(\varphi_R(z))\cdotp\nabla_y\left(\psi_0(z,t)\right)\,\mathrm{dz}\,\mathrm{dt}+\int_0^\tau\int_{\mathcal{B}}u(z,t)\varphi_R(z)\Delta_{\mathcal{G}}\left(\psi_0(z,t)\right)\,\mathrm{dz}\,\mathrm{dt}\nonumber\\
		&&+\int_0^\tau\int_{\mathcal{B}}u(z,t)\varphi_R(z)\partial_t\psi_0(z,t)\,\mathrm{dz}\,\mathrm{dt}.
	\end{eqnarray*}
	By denoting
	$$\mathcal{C}_1=\{z=(x,y)\in \R^{N+k};\,\,R\leq |x|\leq 2R,\,|y|\leq 2R^3\},$$
	$$\mathcal{C}_2=\{z=(x,y)\in \R^{N+k};\,\, |x|\leq 2R,\,R^3\leq |y|\leq 2R^3\},$$
	and using 
	$$\Delta_{\mathcal{G}}\left(\varphi_R(z)\right)= \varphi_2(y)\Delta_{x}\left(\varphi_1(x)\right)+ |x|^2\varphi_1(x)\Delta_{y}\left(\varphi_2(y)\right),$$
	we get
	\begin{eqnarray}
		&&\int_{\mathcal{B}}u(\tau,z)\varphi_R(z)\psi_0(\tau,z)\,\mathrm{dz}-\int_{\mathcal{B}}u_0(z)\varphi_R(z)\psi_0(z,0)\,\mathrm{dz}\nonumber\\
		&&=k_1\Gamma(\alpha)\int_0^\tau\int_{\mathcal{B}}I^\alpha_{0|t}(|u|^{p_1-1}u)\varphi_R(z)\psi_0(z,t)\,\mathrm{dz}\,\mathrm{dt}+k_2\int_0^\tau\int_{\mathcal{B}}|u|^{p_2-1}u\varphi_R(z)\psi_0(z,t)\,\mathrm{dz}\,\mathrm{dt}\nonumber\\
		&&+ \int_0^\tau\int_{\mathcal{C}_1}u(z,t)\psi_0(z,t)\varphi_2(y)\Delta_{x}\left(\varphi_1(x)\right)\,\mathrm{dz}\,\mathrm{dt} +  \int_0^\tau\int_{\mathcal{C}_2}u(z,t)\psi_0(z,t)|x|^2\varphi_1(x)\Delta_{y}\left(\varphi_2(y)\right)\,\mathrm{dz}\,\mathrm{dt}\nonumber\\
		&&+\int_0^\tau\int_{\mathcal{C}_1}u(z,t)\nabla_x(\varphi_R(z))\cdotp\nabla_x\left(\psi_0(z,t)\right)\,\mathrm{dz}\,\mathrm{dt} +  \int_0^\tau\int_{\mathcal{C}_2}u(z,t)|x|^2\nabla_y(\varphi_R(z))\cdotp\nabla_y\left(\psi_0(z,t)\right)\,\mathrm{dz}\,\mathrm{dt}\nonumber\\
		&&+\int_0^\tau\int_{\mathcal{B}}u(z,t)\varphi_R(z)\Delta_{\mathcal{G}}\left(\psi_0(z,t)\right)\,\mathrm{dz}\,\mathrm{dt}+\int_0^\tau\int_{\mathcal{B}}u(z,t)\varphi_R(z)\partial_t\psi_0(z,t)\,\mathrm{dz}\,\mathrm{dt}. \label{50}
	\end{eqnarray}
	On the other hand,
	$$I_1:=\left|\int_0^\tau\int_{\mathcal{C}_1}u(z,t)\psi_0(z,t)\varphi_2(y)\Delta_{x}\left(\varphi_1(x)\right)\,\mathrm{dz}\,\mathrm{dt}\right| \leq \int_0^\tau\int_{\mathcal{C}_1} |u(z,t)| |\psi_0(z,t)|\varphi_2(y)\left|\Delta_{x}\left(\varphi_1(x)\right)\right|\,\mathrm{dz}\,\mathrm{dt}.$$
	Note that, as $\Phi\leq1$ and $\Phi\in C^\infty$ on $\mathcal{C}_1$, we can easily see that
	$$\left|\Delta_x\varphi_1\right|\leq\,C\,{R^{-2}},\qquad\hbox{on}\,\,\mathcal{C}_1,$$
	and therefore
	\begin{align*}
		I_1&\leq C\,R^{-2}\int_0^\tau\int_{\mathcal{C}_1} |u(z,t)| |\psi_0(z,t)|\,\mathrm{dz}\,\mathrm{dt}\\
		& \leq C\,R^{-2}\tau\sup_{t\in[0,\tau]}\|u(t,\cdotp)\|_{L^\infty( \R^{N+k})}\sup_{t\in[0,\tau]}\int_{ \R^{N+k}}|\psi_0(t,\eta)|\,d\eta,\end{align*}
	this implies, using \eqref{49}, that
	\begin{equation}\label{46}
		I_1\longrightarrow 0,\qquad \hbox{when}\,\,R\rightarrow+\infty.
	\end{equation}
	Similarly, using the fact that
	$$|x|^2\left|\Delta_y\varphi_2\right|\leq\,C\,{R^{-4}},\qquad\hbox{on}\,\,\mathcal{C}_2,$$
	we infer that
	\begin{equation}\label{461}
		I_2:=\int_0^\tau\int_{\mathcal{C}_2}u(z,t)\psi_0(z,t)|x|^2\varphi_1(x)\Delta_{y}\left(\varphi_2(y)\right)\,\mathrm{dz}\,\mathrm{dt}\longrightarrow 0,\qquad \hbox{when}\,\,R\rightarrow+\infty.
	\end{equation}
	In addition, 
	\begin{align*}I_3:&=\int_0^\tau\int_{\mathcal{C}_1}u(z,t)\nabla_x(\varphi_R(z))\cdotp\nabla_x\left(\psi_0(z,t)\right)\,\mathrm{dz}\,\mathrm{dt} \\&\leq\int_0^\tau\int_{\mathcal{C}_1} |u(z,t)| \left|\nabla_x(\varphi_R(z))\right| \left|\nabla_x\left(\psi_0(z,t)\right)\right|\,\mathrm{dz}\,\mathrm{dt}.\end{align*}
	As $\Phi\leq1$ and $\Phi\in C^\infty$ on $\mathcal{C}$, we can easily see that
	$$\left|\nabla_x\varphi_1\right|\leq\,C\,{R^{-1}},\qquad\hbox{on}\,\,\mathcal{C}_1,$$
	and therefore
	\begin{align*}I_3&\leq C\,R^{-1}\int_0^\tau\int_{\mathcal{C}_1} |u(z,t)|  \,|\nabla_x\psi_0(z,t)|\,dz\,dt\\&\leq C\,R^{-1}\tau\sup_{t\in[0,\tau]}\|u(t,\cdotp)\|_{L^\infty( \R^{N+k})}\sup_{t\in[0,\tau]}\int_{ \R^{N+k}}|\nabla_x\psi_0(z,t)|\,dz,\end{align*}
	this implies, using \eqref{49}, that
	\begin{equation}\label{47}
		I_3\longrightarrow 0,\qquad \hbox{when}\,\,R\rightarrow+\infty.
	\end{equation}
	Similarly, using the fact that
	$$|x|^2\left|\nabla_y\varphi_2\right|\leq\,C\,{R^{-1}},\qquad\hbox{on}\,\,\mathcal{C}_2,$$
	we infer that
	\begin{equation}\label{471}
		I_4:=\int_0^\tau\int_{\mathcal{C}_2}u(z,t)|x|^2\nabla_y(\varphi_R(z))\cdotp\nabla_y\left(\psi_0(z,t)\right)\,\mathrm{dz}\,\mathrm{dt}\longrightarrow 0,\qquad \hbox{when}\,\,R\rightarrow+\infty.
	\end{equation}
	Finally, letting $R\longrightarrow+\infty$ in \eqref{50} and using \eqref{49}, \eqref{46}-\eqref{471} together with Lebesgue's dominated convergence theorem we conclude the result.
\end{proof}


\subsection{The Cauchy problem (\ref{*})  without memory}

Here, we will consider the case $k_1=0,\ k_2=1$,
\begin{equation}\label{eq1}
	\left\{\begin{array}{ll}
		\,\, \displaystyle {u_{t}-\Delta_{\mathcal{G}}
			u =|u|^{p-1}u,} &\displaystyle {(z,t)\in {\mathbb{R}}^{N+k}\times (0,\infty)},\\
		{}\\
		\displaystyle{u(z,0)=  u_0(z),\qquad\qquad}&\displaystyle{z \in {\mathbb{R}}^{N+k},}
	\end{array}
	\right.
\end{equation} 
where $p>1$.

\begin{theorem}\label{T1}
	Let $u_0\in L^1(\R^{N+k})$ be such that $\displaystyle \int_{\R^{N+k}}u_0(z)\,\mathrm{dz}>0$. If
	$$
	p\leq1+\frac{2}{N+2k}:=p_c,
	$$
	then there is no global nonnegative weak solution to \eqref{eq1}.
\end{theorem}
\begin{proof} The proof is by contradiction. Suppose that $u$ is a global nonnegative weak
	solution to \eqref{eq1}, that is, \eqref{weaksolution} holds for all $T\gg1$, with $k_1=0,\ k_2=1$. Let us choose 
	$$\varphi_	T(z,t)=\left(\varphi_1(x)\right)^\ell \left(\varphi_2(y)\right)^\ell\left(\varphi_3(t)\right)^\ell,\qquad z=(x,y)\in\mathbb{R}^{N+k},\,\,t\geq 0,$$
	with
	$$\varphi_1(x):=\Phi\left(\abs{x}/T^{1/2}\right),\quad \varphi_2(y):=\Phi\left(\abs{y}/T\right),\quad \varphi_3(t):=\eta\left(t/T\right),$$
	where $\ell= 2p/(p-1)$,  and $\Phi,\eta\in \mathcal{C}^\infty(\mathbb{R})$ are smooth non-increasing functions 
	satisfying 
	$$\mathbbm{1}_{(-\infty,1]} \leq \Phi\leq \mathbbm{1}_{(-\infty,2]}\quad\hbox{and}\quad \mathbbm{1}_{(-\infty,\frac{1}{2}]} \leq \eta\leq \mathbbm{1}_{(-\infty,1]}.$$
	Then
	\begin{eqnarray}\label{weak1}
		&{}&\int_{\Omega}u_0(x,y)\left(\varphi_1(x)\right)^\ell \left(\varphi_2(y)\right)^\ell\,dx\,dy+\int_{\Omega_T}|u(z,t)|^{p}\varphi(z,t)\,\mathrm{dz}\,\mathrm{dt}\nonumber\\
		&{}&=-\int_{\Omega_T}u(z,t)\varphi_t(z,t)\,\mathrm{dz}\,\mathrm{dt} -\int_{\Omega_T}u(z,t)\Delta_{\mathcal{G}}\varphi(z,t)\,\mathrm{dz}\,\mathrm{dt}\nonumber\\
		&{}&=: I_1+I_2,
	\end{eqnarray}
	where
	\[
	\Omega_T:=[0,T]\times\Omega\;\;\text{for}\;\; \Omega=\left\{ z=(x,y)\in{\mathbb{R}}^{N+k}\hspace{2
		mm};\hspace{2mm}|x|\leq 2 T^{1/2},\,|y|\leq 2 T\right\}.
	\]
	\noindent In order to estimate the right-hand side of \eqref{weak1}, we introduce the term 
	$\varphi^{1/p}\varphi^{-1/p}$ in $I_1$, and we use Young's inequality to obtain
	\begin{align}\label{I1}
		I_1\leqslant\frac{1}{4} \int_{\Omega_T}|u(z,t)|^{p}\varphi(z,t)\,\mathrm{dz}\,\mathrm{dt}+C\int_{\Omega_T}\varphi_1^\ell(x)\varphi_2^\ell(y)\varphi_3^{p^\prime}(t)\left|\partial_t \varphi_3(t)\right|^{p'}\mathrm{d}z\,\mathrm{d}t.
	\end{align}
	To estimate $I_2$,  notice that 
	\begin{align}\label{IneqA}
		\Delta_{\mathcal{G}}\varphi(z,t) &=\ell(\ell-1)\varphi^\ell_2(y)\varphi^\ell_3(t)\varphi^{\ell-2}_1(x)|\nabla_x\varphi_1(x)|^2+\ell\varphi^\ell_2(y)\varphi^\ell_3(t)\varphi^{\ell-1}_1(x)\Delta_x\varphi_1(x) \\
		&{}\quad +\,\ell(\ell-1)|x|^2\varphi^\ell_1(x)\varphi^\ell_3(t)\varphi^{\ell-2}_2(y)|\nabla_y\varphi_2(y)|^2+\ell|x|^2\varphi^\ell_1(x)\varphi^\ell_3(t)\varphi^{\ell-1}_2(y)\Delta_y\varphi_2(y) .\nonumber
	\end{align}
	Hence
	\begin{eqnarray}\label{I2}
		I_2&\leqslant&\frac{1}{4} \int_{\Omega_T}|u(z,t)|^{p}\varphi(z,t)\,\mathrm{dz}\,\mathrm{dt}+C\int_{\Omega_T}\varphi_2^\ell(y)\varphi_3^\ell(t)\left|\nabla_x \varphi_1(x)\right|^{2p'}\mathrm{d}z\,\mathrm{d}t\notag\\
		&{}&+\,C\int_{\Omega_T}\varphi_2^\ell(y)\varphi_3^\ell(t)\varphi_1^{p'}(x)\left|\Delta_x \varphi_1(x)\right|^{p'}\mathrm{d}z\,\mathrm{d}t\notag \\
		&{}&+C\int_{\Omega_T}|x|^{2p'}\varphi_1^\ell(x)\varphi_3^\ell(t)\left|\nabla_y \varphi_2(y)\right|^{2p'}\mathrm{d}z\,\mathrm{d}t\notag\\
		&{}&+\,C\int_{\Omega_T}|x|^{2p'}\varphi_1^\ell(x)\varphi_3^\ell(t)\varphi_2^{p'}(y)\left|\Delta_y \varphi_2(y)\right|^{p'}\mathrm{d}z\,\mathrm{d}t .
	\end{eqnarray}
	Combining \eqref{I1}, \eqref{I2}, and \eqref{weak1}, we arrive at
	\begin{eqnarray}\label{weak2}
		&{}&\int_{\Omega}u_0(x,y)\left(\varphi_1(x)\right)^\ell \left(\varphi_2(y)\right)^\ell\,dx\,dy+\frac{1}{2}\int_{\Omega_T}|u(z,t)|^{p}\varphi(z,t)\,\mathrm{dz}\,\mathrm{dt}\notag\\
		&{}&\leq C\int_{\Omega_T}\varphi_1^\ell(x)\varphi_2^\ell(y)\varphi_3^{p^\prime}(t)\left|\partial_t \varphi_3(t)\right|^{p'}\mathrm{d}z\,\mathrm{d}t+C\int_{\Omega_T}\varphi_2^\ell(y)\varphi_3^\ell(t)\left|\nabla_x \varphi_1(x)\right|^{2p'}\mathrm{d}z\,\mathrm{d}t\notag\\
		&{}&+\,C\int_{\Omega_T}\varphi_2^\ell(y)\varphi_3^\ell(t)\varphi_1^{p'}(x)\left|\Delta_x \varphi_1(x)\right|^{p'}\mathrm{d}z\,\mathrm{d}t+C\int_{\Omega_T}|x|^{2p'}\varphi_1^\ell(x)\varphi_3^\ell(t)\left|\nabla_y \varphi_2(y)\right|^{2p'}\mathrm{d}z\,\mathrm{d}t\notag\\
		&{}&+\,C\int_{\Omega_T}|x|^{2p'}\varphi_1^\ell(x)\varphi_3^\ell(t)\varphi_2^{p'}(y)\left|\Delta_y \varphi_2(y)\right|^{p'}\mathrm{d}z\,\mathrm{d}t .
	\end{eqnarray}
	Now, we introduce the scaled variables 
	$$\tau_1=T^{-1/2}x,\quad \tau_2=T^{-1}y,\quad \tau_3=T^{-1}t.$$ 
	We get
	\begin{equation}\label{weak3}
		\int_{\Omega}u_0(x,y)\left(\varphi_1(x)\right)^\ell \left(\varphi_2(y)\right)^\ell\,dx\,dy+\frac{1}{2}\int_{\Omega_T}|u(z,t)|^{p}\varphi(z,t)\,\mathrm{dz}\,\mathrm{dt}\leq C\,T^{-p'+\frac{N}{2}+k+1},
	\end{equation}
	which implies that
	\begin{equation}\label{weak4}
		\int_{\Omega}u_0(x,y)\left(\varphi_1(x)\right)^\ell \left(\varphi_2(y)\right)^\ell\,dx\,dy\leq C\,T^{-\delta} ,
	\end{equation}
	with $\delta:=p'-\frac{N}{2}-k-1$. Since $p\leq p_c\Leftrightarrow \delta\geq0$, we have to distinguish two cases:\\
	
	\noindent $\bullet$ The subcritical case $p<p_c$. Passing to the limit in \eqref{weak4} as $T\rightarrow\infty$, and using Lebesgue dominated convergence theorem together with $u_0\in L^1(\mathbb{R}^{N+k})$, we get
	\[
	0<\int_{\Omega}u_0(x,y)\,dx\,dy=\lim_{T\rightarrow\infty}\int_{\Omega}u_0(x,y)\left(\varphi_1(x)\right)^\ell \left(\varphi_2(y)\right)^\ell\,dx\,dy\leq 0;
	\]
	contradiction.\\
	
	\noindent $\bullet$ The critical case $p=p_c$. Using inequality
	\eqref{weak3}  with $T\rightarrow\infty $ and taking into account the fact that $p=p_c$, we have
	\begin{equation}\label{regularity}
		u\in L^p((0,\infty),L^p(\mathbb{R}^{N+k})).
	\end{equation}
	On the other hand, by repeating the same calculation as above by using H\"older's inequality instead of Young's inequality, we arrive at
	\begin{eqnarray*}
		&{}&\int_{\Omega}u_0(x,y)\left(\varphi_1(x)\right)^\ell \left(\varphi_2(y)\right)^\ell\,dx\,dy+\int_{\Omega_T}|u(z,t)|^{p}\varphi(z,t)\,\mathrm{dz}\,\mathrm{dt}\\
		&{}&\leq \left(\int_{T/2}^T\int_{\Omega}|u(z,t)|^{p}\varphi(z,t)\,\mathrm{dz}\,\mathrm{dt}\right)^{1/p}\left(\int_{\Omega_T}\varphi_1^\ell(x)\varphi_2^\ell(y)\varphi_3^{p^\prime}(t)\left|\partial_t \varphi_3(t)\right|^{p'}\mathrm{d}z\,\mathrm{d}t\right)^{1/p'}\\
		&{}&+\left(\int_{\Sigma_T}|u(z,t)|^{p}\varphi(z,t)\,\mathrm{dz}\,\mathrm{dt}\right)^{1/p}\left(\int_{\Omega_T}\varphi_2^\ell(y)\varphi_3^\ell(t)\left|\nabla_x \varphi_1(x)\right|^{2p'}\mathrm{d}z\,\mathrm{d}t\right)^{1/p'}\notag\\
		&{}&+\left(\int_{\Sigma_T}|u(z,t)|^{p}\varphi(z,t)\,\mathrm{dz}\,\mathrm{dt}\right)^{1/p}\left(\int_{\Omega_T}\varphi_2^\ell(y)\varphi_3^\ell(t)\varphi_1^{p'}(x)\left|\Delta_x \varphi_1(x)\right|^{p'}\mathrm{d}z\,\mathrm{d}t\right)^{1/p'}\\
		&{}&+\left(\int_{\Sigma_T}|u(z,t)|^{p}\varphi(z,t)\,\mathrm{dz}\,\mathrm{dt}\right)^{1/p}\left(\int_{\Omega_T}|x|^{2p'}\varphi_1^\ell(x)\varphi_3^\ell(t)\left|\nabla_y \varphi_2(y)\right|^{2p'}\mathrm{d}z\,\mathrm{d}t\right)^{1/p'}\notag\\
		&{}&+\left(\int_{\Sigma_T}|u(z,t)|^{p}\varphi(z,t)\,\mathrm{dz}\,\mathrm{dt}\right)^{1/p}\left(\int_{\Omega_T}|x|^{2p'}\varphi_1^\ell(x)\varphi_3^\ell(t)\varphi_2^{p'}(y)\left|\Delta_y \varphi_2(y)\right|^{p'}\mathrm{d}z\,\mathrm{d}t\right)^{1/p'}
	\end{eqnarray*}
	where
	$$\Sigma_T:=[0,T]\times\Sigma\;\;\text{for}\;\; \Sigma=\left\{ z=(x,y)\in{\mathbb{R}}^{N+k}\hspace{2
		mm};\hspace{2mm}T^{1/2}\leq |x|\leq 2 T^{1/2},\,T\leq |y|\leq 2 T\right\}.$$
	Using the same change of variable as above and taking into account the fact that $p=p_c$, we get
	\begin{align*}
		\int_{\Omega}u_0(x,y)\left(\varphi_1(x)\right)^\ell \left(\varphi_2(y)\right)^\ell\,dx\,dy  \leq & C \left(\int_{T/2}^T\int_{\Omega}|u(z,t)|^{p}\varphi(z,t)\,\mathrm{dz}\,\mathrm{dt}\right)^{1/p} \\
		&+C\left(\int_{\Sigma_T}|u(z,t)|^{p}\varphi(z,t)\,\mathrm{dz}\,\mathrm{dt}\right)^{1/p}.
	\end{align*}
	Recall that $\varphi_1, \varphi_2$ depend on $T$, so we rewrite them, respectively, as $\varphi_{1,T}, \varphi_{2,T}$. Passing to the limit as $T\rightarrow\infty$, and using \eqref{regularity}, we conclude that
	\[
	0<\int_{\mathbb{R}^{N+k}}u_0(x,y)\,dx\,dy=\lim_{T\rightarrow\infty}\int_{\Omega}u_0(x,y)\left(\varphi_{1,T}(x)\right)^\ell \left(\varphi_{2,T}(y)\right)^\ell\,dx\,dy\leq 0.
	\]
	We reached a contradiction.
\end{proof}


\subsection{The Cauchy problem (\ref{*}) with memory term}
In this subsection, we focus on the case $k_1=1$ and $k_2=0$,
\begin{equation}\label{eq2}
	\left\{\begin{array}{ll}
		\,\, \displaystyle {u_{t}-\Delta_{\mathcal{G}}
			u =\int_0^t(t-s)^{-\gamma}\abs{u}^{p-1}u(s)\,ds,} &\displaystyle {(z,t)\in {\mathbb{R}}^{N+k}\times (0,\infty),}\\
		{}\\
		\displaystyle{u(z,0)=  u_0(z),\qquad\qquad}&\displaystyle{z\in {\mathbb{R}}^{N+k},}
	\end{array}
	\right.
\end{equation} 
where $0\leq \gamma<1$, and $p>1$.

\begin{theorem}\label{T2} Assume $0< \gamma<1$, and let $u_0\in L^1(\R^{N+k})$ be such that $\displaystyle \int_{\R^{N+k}}u_0(z)\,\mathrm{dz}>0$. If
	$$p\leq1+\frac{2(2-\gamma)}{N+2k-2+2\gamma}:=p_0\quad\hbox{or}\quad
	p<\frac{1}{\gamma},$$
	then problem \eqref{eq2} admits no global nonnegative weak solution.
\end{theorem}
To complete the case of $p=1/\gamma$, we have
\begin{theorem}\label{T3} Assume $0\leq \gamma<1$, and let $u_0\in C_0(\R^{N+k})$ be such that $u_0\geq0$ and $u_0\not\equiv0$. If
	$$p\gamma \leq 1,$$
	then problem \eqref{eq2} admits no global nonnegative weak solution $u\in C([0,\infty),C_0(\R^{N+k}))$.
\end{theorem}

\begin{proof}[Proof of Theorem \ref{T2}]

	The proof is by contradiction. Suppose that $u$ is a global nonnegative weak
	solution to \eqref{eq2}, then,  \eqref{weaksolution} holds for all $T\gg1$, with $k_1=1$ and $k_2=0$. Let $R$ and $T$ be large parameters in $(0,\infty)$. Let us choose 
	$$\varphi(z,t)=D^\alpha_{t|T}\left(\widetilde{\varphi}(z,t)\right):=D^\alpha_{t|T}\left(\left(\varphi_1(x)\right)^\ell \left(\varphi_2(y)\right)^\ell\left(\varphi_3(t)\right)^\ell\right),\qquad z=(x,y)\in\mathbb{R}^{N+k},\,\,t\geq 0,$$
	with $\ell= 2p/(p-1)$,
	$$\varphi_1(x):=\Phi\left(\abs{x}/R^{1/2}\right),\quad \varphi_2(y):=\Phi\left(\abs{y}/R\right),\quad \varphi_3(t):=\left(1-t/T\right)_+,$$
	where $\Phi\in \mathcal{C}^\infty(\mathbb{R})$ is a smooth non-increasing function 
	satisfying $\mathbbm{1}_{(-\infty,1]} \leq \Phi\leq \mathbbm{1}_{(-\infty,2]}$.
	Thus
	\begin{eqnarray}\label{weak5}
		&&\int_{\Omega}u_0(x,y)\left(\varphi_1(x)\right)^\ell \left(\varphi_2(y)\right)^\ell D^\alpha_{t|T}(\varphi_3)(0)\,dx\,dy + \Gamma(\alpha) \int_{\Omega_T} I^\alpha_{0|t}(u^{p}) \left(D^{\alpha}_{t|T}\widetilde{\varphi}(z,t)\right)\,\mathrm{dz}\,\mathrm{dt}\nonumber\\
		&&=-\int_{\Omega_T} u(z,t) \left(\partial_tD^{\alpha}_{t|T}\widetilde{\varphi}(z,t)\right)\,\mathrm{dz}\,\mathrm{dt} -\int_{\Omega_T}u(z,t) \left(\Delta_{\mathcal{G}}D^{\alpha}_{t|T}\widetilde{\varphi}(z,t) \right)\,\mathrm{dz}\,\mathrm{dt}
	\end{eqnarray}
	where
	\[
	\Omega_T:=[0,T]\times\Omega\;\;\text{for}\;\; \Omega=\left\{ z=(x,y)\in{\mathbb{R}}^{N+k}\hspace{2
		mm};\hspace{2mm}|x|\leq 2 R^{1/2},\,|y|\leq 2 R\right\}.
	\]
	Furthermore, using \eqref{IP}, \eqref{Int3}, \eqref{Int4}, and \eqref{Int6} in  \eqref{weak5}, we obtain
	\begin{eqnarray}
		&&C\,T^{-\alpha}\int_{\Omega}u_0(x,y)\left(\varphi_1(x)\right)^\ell \left(\varphi_2(y)\right)^\ell \,dx\,dy+\Gamma(\alpha)\int_{\Omega_T} u^p(z,t) \widetilde{\varphi}(z,t)\,\mathrm{dz}\,\mathrm{dt} \nonumber\\
		&&=\int_{\Omega_T}u(z,t) \left(D^{1+\alpha}_{t|T}\widetilde{\varphi}(z,t)\right) \,\mathrm{dz}\,\mathrm{dt} -\int_{\Omega_T}u(z,t) \left(\Delta_{\mathcal{G}}D^{\alpha}_{t|T}\widetilde{\varphi}(z,t)\right)\,\mathrm{dz}\,\mathrm{dt} \nonumber\\
		&&=: J_1+J_2. \label{weak6}
	\end{eqnarray}
	\noindent In order to estimate $J_1$ and $J_2$, we introduce the term 
	$\widetilde{\varphi}^{1/p}\widetilde{\varphi}^{-1/p}$ and we use Young's inequality to obtain
	\begin{align}\label{J1}
		J_1\leqslant\frac{\Gamma(\alpha)}{4} \int_{\Omega_T}u^{p}(z,t)\widetilde{\varphi}(z,t)\,\mathrm{dz}\,\mathrm{dt}+C\int_{\Omega_T}\varphi_1^\ell(x)\varphi_2^\ell(y)\varphi_3^{-\ell p^\prime/p}(t)\left|D^{1+\alpha}_{t|T} \varphi^\ell_3(t)\right|^{p'}\mathrm{d}z\,\mathrm{d}t.
	\end{align}
	To estimate $J_2$, we apply identity \eqref{IneqA} and again the Young's inequality in the following way
	\begin{eqnarray}\label{J2}
		J_2&\leqslant&\frac{\Gamma(\alpha)}{4} \int_{\Omega_T}u^{p}(z,t)\widetilde{\varphi}(z,t)\,\mathrm{dz}\,\mathrm{dt}+C\int_{\Omega_T}\varphi_2^\ell(y)\varphi_3^{-\ell p^\prime/p}(t)\left|D^{\alpha}_{t|T} \varphi^\ell_3(t)\right|^{p'}\left|\nabla_x \varphi_1(x)\right|^{2p'}\mathrm{d}z\,\mathrm{d}t\notag\\
		&{}&+\,C\int_{\Omega_T}\varphi_2^\ell(y)\varphi_1^{p'}(x)\varphi_3^{-\ell p^\prime/p}(t)\left|D^{\alpha}_{t|T} \varphi^\ell_3(t)\right|^{p'}\left|\Delta_x \varphi_1(x)\right|^{p'}\mathrm{d}z\,\mathrm{d}t\nonumber\\
		&&+C\int_{\Omega_T}|x|^{2p'}\varphi_1^\ell(x)\varphi_3^{-\ell p^\prime/p}(t)\left|D^{\alpha}_{t|T} \varphi^\ell_3(t)\right|^{p'}\left|\nabla_y \varphi_2(y)\right|^{2p'}\mathrm{d}z\,\mathrm{d}t\notag\\
		&{}&+\,C\int_{\Omega_T}|x|^{2p'}\varphi_1^\ell(x)\varphi_2^{p'}(y)\varphi_3^{-\ell p^\prime/p}(t)\left|D^{\alpha}_{t|T} \varphi^\ell_3(t)\right|^{p'}\left|\Delta_y \varphi_2(y)\right|^{p'}\mathrm{d}z\,\mathrm{d}t.
	\end{eqnarray}
	Combining \eqref{weak6}, \eqref{J1}, and \eqref{J2}, we arrive at
	\begin{eqnarray}\label{weak7}
		&{}&C\,T^{-\alpha}\int_{\Omega}u_0(x,y)\left(\varphi_1(x)\right)^\ell \left(\varphi_2(y)\right)^\ell\,dx\,dy+\frac{\Gamma(\alpha)}{2}\int_{\Omega_T}u^{p}(z,t)\widetilde{\varphi}(z,t)\,\mathrm{dz}\,\mathrm{dt}\notag\\
		&{}&\leq C\int_{\Omega_T}\varphi_1^\ell(x)\varphi_2^\ell(y)\varphi_3^{-\ell p^\prime/p}(t)\left|D^{1+\alpha}_{t|T} \varphi^\ell_3(t)\right|^{p'}\mathrm{d}z\,\mathrm{d}t\notag\\&&\quad+\,C\int_{\Omega_T}\varphi_2^\ell(y)\varphi_3^{-\ell p^\prime/p}(t)\left|D^{\alpha}_{t|T} \varphi^\ell_3(t)\right|^{p'}\left|\nabla_x \varphi_1(x)\right|^{2p'}\mathrm{d}z\,\mathrm{d}t\notag\\
		&{}&\quad+\,\,C\int_{\Omega_T}\varphi_2^\ell(y)\varphi_1^{p'}(x)\varphi_3^{-\ell p^\prime/p}(t)\left|D^{\alpha}_{t|T} \varphi^\ell_3(t)\right|^{p'}\left|\Delta_x \varphi_1(x)\right|^{p'}\mathrm{d}z\,\mathrm{d}t\notag\\
		&&\quad+\,C\int_{\Omega_T}|x|^{2p'}\varphi_1^\ell(x)\varphi_3^{-\ell p^\prime/p}(t)\left|D^{\alpha}_{t|T} \varphi^\ell_3(t)\right|^{p'}\left|\nabla_y \varphi_2(y)\right|^{2p'}\mathrm{d}z\,\mathrm{d}t\notag\\
		&{}&\quad+\,\,C\int_{\Omega_T}|x|^{2p'}\varphi_1^\ell(x)\varphi_2^{p'}(y)\varphi_3^{-\ell p^\prime/p}(t)\left|D^{\alpha}_{t|T} \varphi^\ell_3(t)\right|^{p'}\left|\Delta_y \varphi_2(y)\right|^{p'}\mathrm{d}z\,\mathrm{d}t.
	\end{eqnarray}
	At this stage, we introduce the scaled variables 
	\begin{equation}\label{changevar}
		\tau_1=R^{-1/2}x,\quad \tau_2=R^{-1}y,\quad \tau_3=T^{-1}t,	
	\end{equation} 
	we get
	\begin{eqnarray}\label{weak8}
		&&C\,T^{-\alpha}\int_{\Omega}u_0(x,y)\left(\varphi_1(x)\right)^\ell \left(\varphi_2(y)\right)^\ell\,dx\,dy+\frac{\Gamma(\alpha)}{2}\int_{\Omega_T}u^{p}(z,t)\widetilde{\varphi}(z,t)\,\mathrm{dz}\,\mathrm{dt}\notag\\
		&&\leq C\,T^{1-(1+\alpha)p'}R^{\frac{N}{2}+k}+C\,T^{1-\alpha p'}R^{-p'+\frac{N}{2}+k}.
	\end{eqnarray}
	We have to distinguish three cases:\\
	
	\noindent $\bullet$ The case $p<p_0$. We divide the proof into two steps. In the first step, we take $R=T$. From \eqref{weak8}, and using the fact that $\varphi_1(x),\varphi_2(y)\leq 1$, we conclude that
	\begin{equation}\label{weak9}
		\int_{\Omega_T}(u(z,t))^{p}\widetilde{\varphi}(z,t)\,\mathrm{dz}\,\mathrm{dt}\leq C\,T^{-\delta}+C\,T^{-\alpha}\|u_0\|_{L^1}.
	\end{equation}
	with $\delta:=(1+\alpha)p'-\frac{N}{2}-k-1$. As $p<p_0$ is equivalent to $\delta>0$, we pass to the limit in \eqref{weak9} as $T\rightarrow\infty$, and using the monotone convergence theorem together with $u_0\in L^1(\mathbb{R}^{N+k})$, we get
	$$0\leq \int_0^\infty\int_{\R^{N+k}}(u(z,t))^{p}\,\mathrm{dz}\,\mathrm{dt}\leq 0.$$
	Therefore $u\equiv 0$ almost everywhere. In the second step, coming back to \eqref{weak6} combined with $u\equiv0$ a.e., we may arrive at
	$$C\,T^{-\alpha}\int_{\Omega}u_0(x,y)\left(\varphi_1(x)\right)^\ell \left(\varphi_2(y)\right)^\ell \,dx\,dy\leq 0,$$
	that is,
	$$\int_{\Omega}u_0(x,y)\left(\varphi_1(x)\right)^\ell \left(\varphi_2(y)\right)^\ell \,dx\,dy\leq 0,\quad\hbox{for all}\,\,T>0.$$
	Passing now to the limit when $T\rightarrow\infty$, and using Lebesgue dominated convergence theorem together with $u_0\in L^1(\mathbb{R}^{N+k})$, we get
	\[
	0<\int_{\R^{N+k}}u_0(x,y)\,dx\,dy=\lim_{T\rightarrow\infty}\int_{\Omega}u_0(x,y)\left(\varphi_{1,T}(x)\right)^\ell \left(\varphi_{2,T}(y)\right)^\ell\,dx\,dy\leq 0 ,
	\]
	a contradiction.\\
	
	\noindent $\bullet$ The case $p=p_0$. We can see from \eqref{weak9} with $T\rightarrow\infty$ and taking into account the fact that $p=p_0$ is equivalent to $\delta=0$, that
	\begin{equation}\label{regularityA}
		u\in L^p((0,\infty),L^p(\mathbb{R}^{N+k})).
	\end{equation}
	On the other hand, by repeating the same calculation performed after \eqref{weak6}, by using H\"older's inequality instead of Young's inequality in $J_2$, we arrive at
	\begin{eqnarray*}
		&{}&C\,T^{-\alpha}\int_{\Omega}u_0(x,y)\left(\varphi_1(x)\right)^\ell \left(\varphi_2(y)\right)^\ell \,dx\,dy+\Gamma(\alpha)\int_{\Omega_T}(u(z,t))^{p}\widetilde{\varphi}(z,t)\,\mathrm{dz}\,\mathrm{dt}\\
		&{}&\leq \frac{\Gamma(\alpha)}{4} \int_{\Omega_T}(u(z,t))^{p}\widetilde{\varphi}(z,t)\,\mathrm{dz}\,\mathrm{dt}+C\int_{\Omega_T}\varphi_1^\ell(x)\varphi_2^\ell(y)\varphi_3^{-\ell p^\prime/p}(t)\left|D^{1+\alpha}_{t|T} \varphi^\ell_3(t)\right|^{p'}\mathrm{d}z\,\mathrm{d}t\\
		&{}&+\left(\int_{\Sigma_T}(u(z,t))^{p}\widetilde{\varphi}(z,t)\,\mathrm{dz}\,\mathrm{dt}\right)^{1/p} \notag\\
		&&\times \left[ \left(\int_{\Omega_T}\varphi_2^\ell(y)\varphi_3^{-\ell p^\prime/p}(t)\left|D^{\alpha}_{t|T} \varphi^\ell_3(t)\right|^{p'}\left|\nabla_x \varphi_1(x)\right|^{2p'}\mathrm{d}z\,\mathrm{d}t\right)^{1/p'} \right.\notag\\
		&& \quad + \left(\int_{\Omega_T}\varphi_2^\ell(y)\varphi_1^{p'}(x)\varphi_3^{-\ell p^\prime/p}(t)\left|D^{\alpha}_{t|T} \varphi^\ell_3(t)\right|^{p'}\left|\Delta_x \varphi_1(x)\right|^{p'}\mathrm{d}z\,\mathrm{d}t\right)^{1/p'}\\
		&& \quad + \left(\int_{\Omega_T}|x|^{2p'}\varphi_1^\ell(x)\varphi_3^{-\ell p^\prime/p}(t)\left|D^{\alpha}_{t|T} \varphi^\ell_3(t)\right|^{p'}\left|\nabla_y \varphi_2(y)\right|^{2p'}\mathrm{d}z\,\mathrm{d}t\right)^{1/p'}\notag\\
		&& \quad + \left. \left(\int_{\Omega_T}|x|^{2p'}\varphi_1^\ell(x)\varphi_2^{p'}(y)\varphi_3^{-\ell p^\prime/p}(t)\left|D^{\alpha}_{t|T} \varphi^\ell_3(t)\right|^{p'}\left|\Delta_y \varphi_2(y)\right|^{p'}\mathrm{d}z\,\mathrm{d}t\right)^{1/p'} \right]
	\end{eqnarray*}
	where
	$$\Sigma_T:=[0,T]\times\Sigma\;\;\text{for}\;\; \Sigma=\left\{ z=(x,y)\in{\mathbb{R}}^{N+k}\hspace{2
		mm};\hspace{2mm}R^{1/2}\leq |x|\leq 2 R^{1/2},\,R\leq |y|\leq 2 R\right\}.$$
	Let $R=TK^{-1}$ with some constants $K\ge 1$ with $K<T$ such that $T$ and $K$ cannot go to infinity simultaneously. Using \eqref{changevar} and taking into account the fact that $p=p_c$, we get
	\begin{eqnarray*}
		&{}&\frac{3\Gamma(\alpha)}{4}\int_{\Omega_T}(u(z,t))^{p}\widetilde{\varphi}(z,t)\,\mathrm{dz}\,\mathrm{dt}\\
		&{}&\leq C\,T^{-\alpha}\|u_0\|_{L^1}+CK^{-\frac{N}{2}-k} +\,C\,K^{1-\frac{N}{2p'}-\frac{k}{p'}}\left(\int_{\Sigma_T}(u(z,t))^{p}\widetilde{\varphi}(z,t)\,\mathrm{dz}\,\mathrm{dt}\right)^{1/p}
	\end{eqnarray*}
	where we have used $\varphi_1(x),\varphi_2(y)\geq 0$. Passing to the limit as $T\rightarrow\infty$, and using \eqref{regularityA}, we conclude that
	$$\int_0^\infty\int_{\R^{N+k}}(u(z,t))^{p}\,\mathrm{dz}\,\mathrm{dt}\leq CK^{-\frac{N}{2}-k}.$$
	Therefore, taking a sufficiently large $K$ we obtain the desired contradiction similarly to the first case.\\
	
	\noindent $\bullet$ The case $p<1/\gamma$. We choose in this case $R<T$ such that $T$ and $R$ cannot go to infinity simultaneously. From \eqref{weak8}, and using the fact that $\varphi_1(x),\varphi_2(y)\leq 1$, we conclude that
	\begin{equation}\label{weak10}
		\int_{\Omega_T}(u(z,t))^{p}\widetilde{\varphi}_T(z,t)\,\mathrm{dz}\,\mathrm{dt}\leq C\,T^{1-(1+\alpha)p'}R^{\frac{N}{2}+k}+C\,T^{1-\alpha p'}R^{-p'+\frac{N}{2}+k}+C\,T^{-\alpha}\|u_0\|_{L^1}.
	\end{equation}
	Here we wrote $\widetilde{\varphi}_T$ to emphasize the dependence of $T$. As $p<1/\gamma\Longleftrightarrow 1-\alpha p'<0$, we pass to the limit in \eqref{weak10} as $T\rightarrow\infty$, and using $u_0\in L^1(\mathbb{R}^{N+k})$, we get
	$$0\leq \int_0^\infty\int_{\R^{N+k}}(u(z,t))^{p}\,\mathrm{dz}\,\mathrm{dt}\leq 0.$$
	Therefore $u\equiv 0$ almost everywhere, and so the desired contradiction can be obtained similarly to the first case.
\end{proof}

\begin{proof}[Proof of Theorem \ref{T3}]

	Assume by contradiction that $u\in C([0,\infty),C_0(\R^{N+k}))$ is a global nonnegative weak solution to \eqref{eq2}. Then, \eqref{weaksolution} holds for all $T\gg1$, with $k_1=1$ and $k_2=0$. Let $T$ be large parameter in $(0,\infty)$, $\psi_T\in C^1([0,T])$ with $\psi_T(T)=0$, and
	$$\varphi(z,t)=\Theta(z)\psi_T(t),\qquad z=(x,y)\in\mathbb{R}^{N+k},\,\,t> 0,$$
	where $\Theta$ is defined  in Lemma \ref{lemma4}. Using Lemma \ref{lemma}, we have
	\begin{eqnarray*}
		&&\int_{\R^{N+k}}u_0(z)\Theta(z) \psi_T(0)\,\mathrm{dz}+\Gamma(\alpha)\int_{Q_T}I^\alpha_{0|t}(u^{p})\Theta(z)\psi_T(t)\,\mathrm{dz}\,\mathrm{dt}\nonumber\\
		&&=-\int_{Q_T}u(z,t)\Theta(z)\,\mathrm{dz}\,\partial_t\psi_T(t)\,\mathrm{dt} -\int_{Q_T}u(z,t)\psi_T(t)\Delta_{\mathcal{G}}\Theta(z)\,\mathrm{dz}\,\mathrm{dt}
	\end{eqnarray*}
	where $Q_T:=[0,T]\times \mathbb{R}^{N+k}$, that is
	\begin{eqnarray*}
		&&-\int_{Q_T}u(z,t)\Theta(z)\,\mathrm{dz}\,\partial_t\psi_T(t)\,\mathrm{dt} -\int_{\R^{N+k}}u_0(z)\Theta(z)\,\mathrm{dz}\,\psi_T(0)\nonumber\\
		&&=\Gamma(\alpha)\int_{Q_T}I^\alpha_{0|t}(u^{p})\Theta(z)\psi_T(t)\,\mathrm{dz}\,\mathrm{dt}+\int_{Q_T}u(z,t)\psi_T(t)\Delta_{\mathcal{G}}\Theta(z)\,\mathrm{dz}\,\mathrm{dt}
	\end{eqnarray*}
	Therefore, by letting $f(t):= \displaystyle\int_{\R^{N+k}}u(z,t)\Theta(z)\,\mathrm{dz}$, we arrive that
	\begin{eqnarray*}
		&&-\int_0^T(f(t)-f(0))\,\partial_t\psi_T(t)\,\mathrm{dt} \nonumber\\
		&&=\int_0^T\int_{\R^{N+k}}\left(\Gamma(\alpha)I^\alpha_{0|t}(u^{p})\Theta(z)+u(z,t)\Delta_{\mathcal{G}}\Theta(z)\right)\mathrm{dz}\,\psi_T(t)\,\mathrm{dt}
	\end{eqnarray*}	
	Due to the arbitrariness of $\psi_T$, we obtain in the sense of distributions	
	\begin{equation}\label{500}
		\frac{d}{dt}(f(t)-f(0)) =\int_{\R^{N+k}}\left(\Gamma(\alpha)I^\alpha_{0|t}(u^{p})\Theta(z)+u(z,t)\Delta_{\mathcal{G}}\Theta(z)\right)\mathrm{dz}.
	\end{equation}	
	In addition, since $(f(t)-f(0))\in C([0,T])$, and 
	$$
	\int_{\R^{N+k}}\left(\Gamma(\alpha)I^\alpha_{0|t}(u^{p})\Theta(z)+u(z,t)\Delta_{\mathcal{G}}\Theta(z)\right)\mathrm{dz}\in C([0,T]),
	$$
	it follows that the equality \eqref{500} holds for $t\in[0,T]$, that is
	$$f'(t) =\int_0^t(t-s)^{-\gamma}\int_{\R^{N+k}} u^p(z,s)\Theta(z)\,\mathrm{dz}\,ds+\int_{\R^{N+k}}u(z,t)\Delta_{\mathcal{G}}\Theta(z)\,\mathrm{dz}.
	$$
	Applying Lemma \ref{lemma4} where $c>0$ is such that $\displaystyle \int_Q\Theta(z)\,\mathrm{dz}=1$ with $\varepsilon=1/(2(N+2)+k)$, that is $\Delta_{\mathcal{G}}\Theta(z)\geq-\Theta(z)$, we conclude that
	$$f'(t)\geq -f(t)+\int_0^t(t-s)^{-\gamma}\int_{\R^{N+k}} u^p(z,s)\Theta(z)\,\mathrm{dz}\,ds.$$
	As $\displaystyle \int_{\R^{N+k}}\Theta (z)\,\mathrm{dz}=1$, we conclude by Jensen’s inequality that
	\begin{equation}\label{int3A}
		f'(t)+f(t)\geq\int_0^t(t-s)^{-\gamma}f^p(s)\,ds,\qquad\hbox{for all}\,\,t\in[0,T].
	\end{equation}
	Since $u_0\geq0$, $u_0\not\equiv0$ implies $f(0)>0$, \eqref{int3A}, $p\gamma\leq 1$, and Lemma \ref{lemma3} $\mathrm{(ii)}$ yield a contradiction. 
\end{proof}

\begin{remark} It follows from \eqref{int3A}, and Lemma \ref{lemma3} $\mathrm{(i)}$ that for any $p >1$, there exists a constant $B> 0$ such that
	if $u_0\geq0$ satisfies
	$$\int_{\R^{N+k}}u_0(x,y)e^{-\varepsilon\sqrt{1+|x|^4+|y|^2}}\,dx\,dy\geq B,$$
	with $\varepsilon=1/(2(N+2)+k)$, then there is no nonnegative global weak solution of \eqref{eq2}.
\end{remark}


\subsection{The Cauchy problem (\ref{*}) with source and memory term}
This subsection is devoted to the case $k_1=k_2=1$,
\begin{equation}\label{eq3}
	\left\{\begin{array}{ll}
		\,\, \displaystyle {u_{t}-\Delta_{\mathcal{G}}
			u =\int_0^t(t-s)^{-\gamma}|u(s)|^{p_1-1}u(s)\,ds+|u(t)|^{p_2-1}u(t),} &\displaystyle {(z,t)\in{\mathbb{R}}^{N+k}\times (0,\infty),}\\
		{}\\
		\displaystyle{u(z,0)=  u_0(z),\qquad\qquad}&\displaystyle{z=(x,y)\in {\mathbb{R}}^{N+k},}
	\end{array}
	\right.
\end{equation} 
where $0\leq \gamma<1$, and $p_1,p_2>1$.


\begin{theorem}\label{T4} Assume $0<\gamma<1$, and let $u_0\in L^1(\R^{N+k})$ be such that $\displaystyle \int_{\R^{N+k}}u_0(z)\,\mathrm{dz}>0$. If
	$$p_1\leq1+\frac{2(2-\gamma)}{N+2k-2+2\gamma}\quad\hbox{or}\quad p_2\leq1+\frac{2}{N+2k}\quad\hbox{or}\quad
	p_1<\frac{1}{\gamma},$$
	then there is no global nonnegative weak solution to \eqref{eq3}.
\end{theorem}

Also, we can complete the case of $p_1=1/\gamma$ with an analogous result to Theorem \ref{T3}.
\begin{theorem}\label{T5} Assume $0\leq \gamma<1$, and let $u_0\in C_0(\R^{N+k})$ be such that $u_0\geq0$ and $u_0\not\equiv0$. If $p_1\gamma\leq 1$,  then
	\eqref{eq3} admits no global nonnegative weak solution $u\in C([0,\infty),C_0(\R^{N+k}))$.
\end{theorem}

\subsubsection{Proof of Theorem \ref{T4}}
The proof combines ideas used to prove Theorems \ref{T1} and \ref{T2}. Indeed, we suppose that $u$ is a global nonnegative weak solution to \eqref{eq3}, that is,  \eqref{weaksolution} holds for all $T\gg1$, with $k_1=k_2=1$. At this stage, we have two cases to distinguish.\\

\noindent $\bullet$ If 
$$p_1\leq1+\frac{2(2-\gamma)}{N+2k-2+2\gamma}\quad\hbox{or}\quad p_1<\frac{1}{\gamma} ,$$
we denote $R$ and $T$ to be large parameters in $(0,\infty)$, and define $\varphi$ as in the proof of Theorem \ref{T2}. So, using the fact that 
$$\displaystyle \int_{\Omega_T}(u(z,t))^{p_2}D^{\alpha}_{t|T}\widetilde{\varphi}(z,t)\,\mathrm{dz}\,\mathrm{dt}\geq0,$$ 
we conclude that
\begin{eqnarray*}
	&&\int_{\Omega}u_0(x,y)\left(\varphi_1(x)\right)^\ell \left(\varphi_2(y)\right)^\ell D^\alpha_{t|T}(\varphi_3)(0)\,dx\,dy+\Gamma(\alpha)\int_{\Omega_T}I^\alpha_{0|t}((u(z,t))^{p_1})D^{\alpha}_{t|T}\widetilde{\varphi}(z,t)\,\mathrm{dz}\,\mathrm{dt}\nonumber\\
	&&\leq -\int_{\Omega_T}u(z,t)\partial_tD^{\alpha}_{t|T}\widetilde{\varphi}(z,t)\,\mathrm{dz}\,\mathrm{dt} -\int_{\Omega_T}u(z,t)\Delta_{\mathcal{G}}D^{\alpha}_{t|T}\widetilde{\varphi}(z,t)\,\mathrm{dz}\,\mathrm{dt}
\end{eqnarray*}
where
\[
\Omega_T:=[0,T]\times\Omega\;\;\text{for}\;\; \Omega=\left\{ z=(x,y)\in{\mathbb{R}}^{N+k}\hspace{2
	mm};\hspace{2mm}|x|\leq 2 R^{1/2},\,|y|\leq 2 R\right\}.
\]
The right-hand side of the above inequality is the same as \eqref{weak5}. Hence, we can apply the same calculations in the proof of Theorem \ref{T2} from inequality \eqref{weak6} until obtain the required contradiction.\\

\noindent $\bullet$ If 
$$p_2\leq1+\frac{2}{N+2k}.$$
Let us choose $\varphi$ as in the proof of Theorem \ref{T1}. Using the fact that $$\displaystyle \int_{\Omega_T}I^\alpha_{0|t}(u(z,t))^{p_1}\varphi(z,t)\,\mathrm{dz}\,\mathrm{dt}\geq0,$$ 
we conclude that
\begin{eqnarray*}
	&{}&\int_{\Omega}u_0(x,y)\left(\varphi_1(x)\right)^\ell \left(\varphi_2(y)\right)^\ell\,dx\,dy +  \int_{\Omega_T}(u(z,t))^{p_2}\varphi(z,t)\,\mathrm{dz}\,\mathrm{dt}\\
	&\leq&-\int_{\Omega_T}u(z,t)\varphi_t(z,t)\,\mathrm{dz}\,\mathrm{dt} -\int_{\Omega_T}u(z,t)\Delta_{\mathcal{G}}\varphi(z,t)\,\mathrm{dz}\,\mathrm{dt}
\end{eqnarray*}
where
\[
\Omega_T:=[0,T]\times\Omega\;\;\text{for}\;\; \Omega=\left\{ z=(x,y)\in{\mathbb{R}}^{N+k}\hspace{2
	mm};\hspace{2mm}|x|\leq 2 T^{1/2},\,|y|\leq 2 T\right\}.
\]
By applying the same calculation as in the proof of Theorem \ref{T1}, we get the required contradiction.\hfill$\square$

\subsubsection{Proof of Theorem \ref{T5}}
The proof is similar to the one of Theorem \ref{T3}. In fact, all the calculations can be repeated but carrying the integral corresponding to the term $|u(t)|^{p_2-1}u(t)$. Then, instead of \eqref{int3A}, we arrive at
\begin{equation}\label{int4A}
	f'(t)+f(t)\geq\int_0^t(t-s)^{-\gamma}f^{p_1}(s)\,ds +  f^{p_2}(t) ,\qquad\hbox{for all}\,\,t\in[0,T].
\end{equation}
Since $f^{p_2}(t)\geq0$, the proof is finished exactly as in Theorem \ref{T3}.
\hfill$\square$\\


\section{Final Remarks}

In this paper we proved the non-existence of nonnegative solutions for \eqref{*}, with $k_1,k_2\geq0$. As we already said in the Introduction, the results in the previous section combined with the results in \cite{Fino-Viana-25.1} provide the critical Fujita exponents in this case. Next, we finish the paper by providing sufficient conditions for the non-existence of solutions in the case of mixed-sign nonlinear terms: Source vs Absorption. In these cases, we do not know if the results are optimal in the sense of the Fujita critical exponents. However, they give insights and leave interesting open problems for future researches.

By assuming that $k_1=1$ and $k_2=-1$, we write
\begin{equation}\label{eq4}
	\left\{\begin{array}{ll}
		\,\, \displaystyle {u_{t}-\Delta_{\mathcal{G}}
			u =\int_0^t(t-s)^{-\gamma}|u(s)|^{p_1-1}u(s)\,ds-|u(t)|^{p_2-1}u(t),} &\displaystyle {(z,t)\in {\mathbb{R}}^{N+k}\times (0,\infty),}\\
		{}\\
		\displaystyle{u(z,0)=  u_0(z),\qquad\qquad}&\displaystyle{z\in {\mathbb{R}}^{N+k}.}
	\end{array}
	\right.
\end{equation}

By using either the idea in the proof of Theorem 2.3 in \cite{Chen-Fino-23} or in the proof of Theorem 2.3 in \cite{K-Rei-21}, it can shown the following: if $u_0\in L^1(\R^{N+k})$ is such that $\displaystyle \int_{\R^{N+k}}u_0(z)\,\mathrm{dz}>0$, $p_1\gamma\leq 1$, or
$$\left\{
\begin{array}{ll}
	p_1<\frac{\min\left\{p_2;\,\frac{\frac{N}{2}+k+\gamma}{\frac{N}{2}+k-1+\gamma}\right\}\left[1-(1-\gamma)\left(\frac{N}{2}+k\right)\right]_++\min\left\{(1-\gamma)\left(\frac{N}{2}+k\right),1\right\}}{\gamma}&\,\,\hbox{when}\quad p_1>p_2\\\\
	p_1\leq 1+\frac{2(2-\gamma)}{N+2k-2+2\gamma}&\,\,\hbox{when}\quad p_1\leq p_2\\
\end{array}
\right.
$$
then there is no global nonnegative weak solution to \eqref{eq4}. 

In the above result, when $p_1\leq p_2$, we need $u_0\in L^1(\R^{N+k})\cap C_0(\R^{N+k})$. Moreover, in the case $p_1\gamma=1$ with $p_1<p_2$, we can follow the same approach as in Theorem 1.1 (i)-(a) in \cite{Zh-23}.


Next, consider $k_1<0$ and $k_2>0$. Without loss of generality, we may assume that $k_1=-1$ and $k_2=1$. Then, we have
\begin{equation}\label{eq5}
	\left\{\begin{array}{ll}
		\,\, \displaystyle {u_{t}-\Delta_{\mathcal{G}}
			u =|u(t)|^{p_2-1}u(t)-\int_0^t(t-s)^{-\gamma}|u(s)|^{p_1-1}u(s)\,ds,} &\displaystyle {(z,t)\in {\mathbb{R}}^{N+k}\times (0,\infty),}\\
		{}\\
		\displaystyle{u(z,0)=  u_0(z),\qquad\qquad}&\displaystyle{z\in {\mathbb{R}}^{N+k}.}
	\end{array}
	\right.
\end{equation} 


In this case, it can be proved that if $u_0\in L^1(\R^{N+k})\cap L^\infty(\R^{N+k})$ is such that $\displaystyle \int_{\R^{N+k}}u_0(z)\,\mathrm{dz}>0$, $p_2\geq p_1$ and
$$p_1\leq 1+\frac{2(2-\gamma)}{N+2k-2+2\gamma}\,\,\,\hbox{or}\,\,\,p_1<\frac{1}{\gamma}
$$
then there is no global nonnegative weak solution to \eqref{eq5}.


\section*{Acknowledgments}

A. Fino is supported by the Research Group Unit, College of Engineering and Technology, American University of the Middle East. A. Viana is partially CAPES-Humboldt Research Fellowship Programme for Experienced Researchers PE32739979 and is part of Universal Fapitec Project n$^{\circ}$ 019203.01303/2024-1. Part of this research was carried out while A. Viana was visiting the Institut f\"ur Angewandte Mathematik, Universit\"at Ulm. He would like to thank the IAA for the hospitality.

\bibliographystyle{abbrv}
\bibliography{ref}

\end{document}